\documentclass[a4paper,12pt,twoside,leqno,final]{amsart}
\usepackage{amsmath}
\usepackage{amssymb}

\newtheorem{thm}{Theorem}[section]

\newtheorem{prop}[thm]{Proposition}

\newcommand{\C}{{\mathbb C}}
\newcommand{\D}{{\mathbb D}}

\newcommand{\T}{{\mathbb T}}

\newcommand{\f}{\frac}
\newcommand{\ov}{\overline}

\newcommand{\la}{\lambda}
\newcommand{\ze}{\zeta}
\renewcommand{\th}{\theta}
\newcommand{\si}{\sigma}

\numberwithin{equation}{section}

\title[A characterization of M\"obius transformations]
{A characterization of M\"obius transformations\\
\quad\\
{\it Une caract\'erisation des transformations de M\"obius}}
\author{Konstantin M. Dyakonov}
\address{ICREA and Universitat de Barcelona, Departament de Matem\`atica 
Aplicada i An\`alisi, Gran Via 585, E-08007 Barcelona, Spain}
\email{konstantin.dyakonov@icrea.cat}
\thanks{Supported in part by grant MTM2011-27932-C02-01 from El Ministerio de Ciencia 
e Innovaci\'on (Spain) and grant 2014-SGR-289 from AGAUR (Generalitat de Catalunya).}

\begin{document}
\begin{abstract}
We prove that the derivative $\th'$ of an inner function $\th$ is outer if and only if $\th$ is 
a M\"obius transformation. An alternative characterization involving a reverse Schwarz--Pick type 
estimate is also given. 

\bigskip

\noindent{\bf R\'esum\'e.} Etant donn\'ee une fonction int\'erieure $\th$, on d\'emontre que sa 
d\'eriv\'ee $\th'$ est ext\'erieure si et seulement si $\th$ est une transformation de M\"obius. 
\end{abstract}

\maketitle

\section{Introduction and main result}

Let $H^\infty$ stand for the algebra of bounded holomorphic functions on the disk $\D:=\{z\in\C:|z|<1\}$. 
A function $\th\in H^\infty$ is called {\it inner} if $\lim_{r\to1^-}|\th(r\ze)|=1$ at almost 
every point $\ze$ of the circle $\T:=\partial\D$. Among the nonconstant inner functions, 
the simplest ones are undoubtedly the conformal automorphisms of the disk, also known as 
{\it M\"obius transformations}; these are of the form 
$$\th_{\la,a}(z):=\la\frac{z-a}{1-\bar az}$$ 
for some $\la\in\T$ and $a\in\D$. A calculation shows that 
$$\th'_{\la,a}(z)=\la\frac{1-|a|^2}{(1-\bar az)^2},$$ 
which happens to be an {\it outer} function. (A nonvanishing holomorphic function $f$ on $\D$ is 
said to be outer if $\log|f|$ agrees with the harmonic extension of an integrable function on $\T$.) 
\par In this note, we prove that the property of $\th'$ being outer actually characterizes 
the M\"obius transformations among all inner functions $\th$. 
\par Before stating the result rigorously, we need to recall that the {\it Nevanlinna class} 
$\mathcal N$ (resp., the {\it Smirnov class} $\mathcal N^+$) is formed by the functions that can be 
written as $u/v$, where $u,v\in H^\infty$ and $v$ is zero-free (resp., outer) on $\D$. The reader is 
referred to \cite[Chapter II]{G} for further information on $\mathcal N$ and $\mathcal N^+$, including 
the canonical factorization theorem for functions from these spaces. We also mention the fact that, for 
$\th$ inner, one has $\th'\in \mathcal N$ if and only if $\th'\in\mathcal N^+$; see \cite{AC} for a proof. 
In what follows, we are forced to require that $\th'$ be in $\mathcal N$ (or $\mathcal N^+$), since this 
is apparently the weakest natural assumption that allows us to speak of the inner-outer factorization 
for $\th'$. 

\begin{thm}\label{thm:charm} Let $\th$ be a nonconstant inner function with $\th'\in\mathcal N$. 
Then $\th'$ is outer if and only if $\th$ is a M\"obius transformation. 
\end{thm}

\par In some special cases, the fact that the derivative of a non-M\"obius inner function will have 
a nontrivial inner part may be obvious or due to known results. First of all, $\th'$ will certainly vanish 
at the multiple zeros of $\th$, if any. Secondly, a result of Ahern and Clark (see \cite[Corollary 4]{AC}) 
tells us that the singular factor of $\th$, if existent, gets inherited by $\th'$, provided that the 
latter function is in $\mathcal N$. Thus, in a sense, singular factors can be thought of as responsible 
for the (boundary) zeros of infinite multiplicity. Thirdly, if $\th$ is a finite Blaschke product with 
at least two zeros, then $\th'$ is sure to have zeros in $\D$ (see \cite{W} for a more precise information 
on the location of these), so $\th'$ will again be non-outer. The remaining case, where $\th$ is an infinite 
Blaschke product with simple zeros, seems however to be new. 

\section{Proof of Theorem \ref{thm:charm}}

To prove the nontrivial part of the theorem, assume that $\th$ is inner and $\th'$ is an outer function 
in $\mathcal N$. 

\par For all $z\in\D$ and almost all $\ze\in\T$, Julia's lemma (see \cite{Car} or \cite[p.\,41]{G}) yields 
\begin{equation}\label{eqn:julia}
\f{|\th(\ze)-\th(z)|^2}{1-|\th(z)|^2}\le|\th'(\ze)|\cdot\f{|\ze-z|^2}{1-|z|^2}, 
\end{equation}
or equivalently, 
\begin{equation}\label{eqn:juliabis}
\f{1-|z|^2}{1-|\th(z)|^2}\cdot\left|\f{1-\ov{\th(z)}\th(\ze)}{1-\ov z\ze}\right|^2\le|\th'(\ze)|.
\end{equation}
Further, we associate with every (fixed) $z\in\D$ the $H^\infty$-function 
\begin{equation}\label{eqn:phiz}
\Phi_z(w):=\f{1-|z|^2}{1-|\th(z)|^2}\cdot\left(\f{1-\ov{\th(z)}\th(w)}{1-\ov zw}\right)^2
\end{equation}
and rewrite \eqref{eqn:juliabis} in the form 
\begin{equation}\label{eqn:estone}
|\Phi_z(\ze)|\le|\th'(\ze)|,\qquad\ze\in\T.
\end{equation}

\par Since $\Phi_z\in H^\infty$ and $\th'$ is outer, the ratio $\psi_z:=\Phi_z/\th'$ will be in $\mathcal N^+$; 
and since, by \eqref{eqn:estone}, $|\psi_z|\le 1$ a.\,e. on $\T$, it follows that $\psi_z$ is 
in $H^\infty$ and has norm at most $1$. In other words, the estimate \eqref{eqn:estone} extends into 
the disk, so that 
$$|\Phi_z(w)|\le|\th'(w)|,\qquad w\in\D.$$ 
In particular, putting $w=z$, we obtain 
\begin{equation}\label{eqn:estatz}
|\Phi_z(z)|\le|\th'(z)|.
\end{equation}
A glance at \eqref{eqn:phiz} reveals that 
$$|\Phi_z(z)|=\Phi_z(z)=\f{1-|\th(z)|^2}{1-|z|^2},$$
and plugging this into \eqref{eqn:estatz} gives 
$$\f{1-|\th(z)|^2}{1-|z|^2}\le|\th'(z)|.$$
In conjunction with the Schwarz--Pick estimate 
\begin{equation}\label{eqn:schp}
|\th'(z)|\le\f{1-|\th(z)|^2}{1-|z|^2}
\end{equation}
(see \cite[Chapter I, Section 1]{G}), this means that we actually have equality in \eqref{eqn:schp}. This last 
fact is known to imply that $\th$ is a M\"obius transformation (see {\it ibid.}), and the proof is complete. 

\section{An alternative characterization and open questions} 

The primary purpose of this note, essentially accomplished by now, can be described as giving a short 
and self-contained proof of a result from \cite{DCMFT}. In that paper, our main concern was a certain 
reverse Schwarz--Pick type inequality for unit-norm $H^\infty$ functions (see also \cite{DSpb} for an 
earlier version); the above characterization of M\"obius transformations was then deduced as a corollary. 
In addition, it was shown in \cite[Section 2]{DCMFT} that, among the nonconstant inner functions $\th$ 
with $\th'\in \mathcal N$, the M\"obius transformations are also characterized by the property that 
\begin{equation}\label{eqn:esteta1}
\eta\left(\f{1-|\th(z)|^2}{1-|z|^2}\right)\le|\th'(z)|,\qquad z\in\D, 
\end{equation}
for some nondecreasing function $\eta:(0,\infty)\to(0,\infty)$. We now improve this last result by 
relaxing the {\it a priori} assumptions on $\th$. In fact, it turns out that we need not restrict our 
attention to inner functions from the start. Instead, we shall verify that $\th$ will have to be inner 
(and with derivative in $\mathcal N$) automatically, under the milder hypotheses below. 

\begin{prop}\label{prop:impr} Let $\th\in H^\infty$ be a nonconstant function with $\|\th\|_\infty\le1$. 
The following conditions are equivalent. 
\par{\rm(i)} $\th$ is a M\"obius transformation. 
\par{\rm(ii)} There is a nondecreasing function $\eta:(0,\infty)\to(0,\infty)$ with 
$\lim_{t\to\infty}\eta(t)=\infty$ making \eqref{eqn:esteta1} true. 
\end{prop} 

\begin{proof} Of course, (i) implies (ii) with $\eta(t)=t$. To prove the nontrivial implication (ii)$\implies$(i), 
observe that 
$$\inf\left\{\f{1-|\th(z)|^2}{1-|z|^2}:\,z\in\D\right\}>0$$ 
(by Schwarz's lemma), and so \eqref{eqn:esteta1} yields $\inf_{z\in\D}|\th'(z)|>0$. Hence $1/\th'\in H^\infty$ 
and $\th'\in \mathcal N$; in particular, $\th'$ has radial limits almost everywhere on $\T$. 
\par Now, if $\ze\in\T$ is a point at which $\lim_{r\to1^-}|\th(r\ze)|<1$, then \eqref{eqn:esteta1} shows that 
$\lim_{r\to1^-}|\th'(r\ze)|=\infty$. Consequently, the set of such $\ze$'s has zero measure on $\T$. It follows 
that $\th$ has radial limits of modulus 1 almost everywhere, and is therefore an inner function. To complete the 
proof, it remains to invoke the above-mentioned result from \cite{DCMFT}. 
\end{proof}

\par We conclude by mentioning two open questions that puzzle us. First, we would like to know which inner 
functions $I$ can be written as $I=\text{\rm inn}(\th')$ (where \lq\lq inn" stands for \lq\lq the inner 
factor of"), as $\th$ ranges over the nonconstant inner functions with $\th'\in\mathcal N$. Does every 
inner $I$ arise in this way? 
\par To pose the other question, let us introduce the notation $\si(I)$ for the {\it boundary spectrum} of 
an inner function $I$. Thus, $\si(I)$ is the smallest closed set $E\subset\T$ such that $I$ is analytic across 
$\T\setminus E$. Now, if $\th$ is inner (and nonconstant) with $\th'\in\mathcal N$, and if $I=\text{\rm inn}(\th')$, 
then it is easy to see that $\si(I)\subset\si(\th)$. Do we actually have $\si(I)=\si(\th)$? An affirmative 
answer seems plausible to us, but so far, we have only succeeded in verifying it under an additional 
hypothesis.

\end{document}